\documentclass[12pt,journal]{IEEEtran}
\onecolumn
\usepackage{mathrsfs}
\usepackage{amsthm}
\usepackage{amsmath}
\usepackage{graphicx}
\usepackage[justification=centering]{caption}
\usepackage{amssymb}
\usepackage{epstopdf}
\usepackage{enumerate}
\usepackage{longtable,tabularx,float}
\usepackage{cite}
\usepackage{mathcomp}
\usepackage{supertabular}
\usepackage{stmaryrd}
\usepackage{algorithm}
\usepackage{algpseudocode}
\usepackage{color}
\usepackage{url}
\usepackage{tikz}
\usepackage{hyperref}
\usepackage{cleveref}
\newtheorem{corollary}{Corollary}[section]

\newtheorem{lemma}{Lemma}[section]

\newtheorem{theorem}{Theorem}[section]

\newtheorem{proposition}{Proposition}[section]
\newtheorem{example}{Example}[section]
\newtheorem{construction}{Construction}[section]

\newtheorem{remark}{Remark}[section]

\renewcommand{\S}{\mathcal {S}}

\newcommand{\Rm}[1]{\uppercase\expandafter{\romannumeral #1\relax}}

\newcommand{\vi}{{\mathbf{i}}}
\newcommand{\vw}{{\mathbf{w}}}
\newcommand{\vx}{{\mathbf{x}}}
\newcommand{\vy}{{\mathbf{y}}}
\newcommand{\vz}{{\mathbf{z}}}

\newcommand{\vp}{{\mathbf{p}}}
\newcommand{\vq}{{\mathbf{q}}}

\newcommand{\Rmnum}[1]{\uppercase\expandafter{\romannumeral #1}}

\numberwithin{equation}{section}

 \usepackage{xcolor}
\definecolor{cream}{RGB}{203, 237, 204}
\pagecolor{cream!90}	

\begin{document}

\title{Improvements on Permutation Reconstruction\\ from Minors}

\author{Yiming~Ma, Wenjie~Zhong~and~Xiande~Zhang
        \thanks{Y. Ma ({\tt mym024@mail.ustc.edu.cn}) is with the School of Cyber Security, University of Science and Technology of China, Hefei, 230026, Anhui, China.}
        \thanks{W. Zhong ({\tt zhongwj@mail.ustc.edu.cn}) is with the School of Mathematical Sciences, University of Science and Technology of China, Hefei, 230026, Anhui, China.}

\thanks{X. Zhang ({\tt drzhangx@ustc.edu.cn}) is with the School of Mathematical Sciences,
University of Science and Technology of China, Hefei, 230026, Anhui, China, and with Hefei National Laboratory, University of Science and Technology of China, Hefei, 230088, Anhui, China.}

\thanks{The research is supported by the National Key Research and
Development Programs of China 2023YFA1010201 and 2020YFA0713100, the NSFC
under Grants No. 12171452 and No. 12231014, and the Innovation Program for Quantum Science and Technology
(2021ZD0302902).}
}

\maketitle

\begin{abstract}
  We study the reconstruction problem of permutation sequences  from their $k$-minors, which are subsequences of length $k$ with entries renumbered by $1,2,\ldots,k$ preserving order.  We prove that the minimum number $k$ such that any permutation of length $n$ can be reconstructed from the multiset of its $k$-minors is between $\exp{(\Omega(\sqrt{\ln n}))}$ and $O(\sqrt{n\ln n})$.  These results imply better bounds of a well-studied parameter $N_d$, which is the smallest number such that any permutation of length  $n\ge N_d$ can be reconstructed by its $(n-d)$-minors. The new bounds are $ d+\exp(\Omega(\sqrt{\ln d}))<N_d<d+O(\sqrt{d\ln d})$ asymptotically, and the previous bounds were  $d+\log_2 d<N_d<d^2/4+2d+4$.

\end{abstract}


\section{Introduction}
Reconstructing a combinatorial object from a limited amount of sub-information is a fundamental
problem in computer science. Based on different combinatorial objects, different reconstruction problems have been widely studied due to their applications in  bioinformatics \cite{acharya2015string,batu2004reconstructing},  information theory \cite{ukkonen1985finding}, and  DNA based data storage \cite{golm2022gapped,yazdi2017portable,gabrys2019unique}.

Permutation reconstruction is a variant of the well-known graph reconstruction problem, which arose from the unsolved conjecture of Ulam \cite{ulam1960collection}: any simple graph with at least three vertices can be determined up to isomorphism by the multiset of all its reduced subgraphs with one vertex deleted.
For permutation reconstruction, one considers reconstructing a permutation from its $k$-minors, that is, subsequences of length $k$ with entries renumbered by $1,2,\ldots,k$ preserving order. The multiset of all its $k$-minors is called its $k$-deck.
In 2006, Smith \cite{smith2006permutation} introduced the notation $N_d$, which is the smallest number such that any permutation of length  $n\ge N_d$ can be reconstructed by its $(n-d)$-deck. Raykova \cite{raykova2006permutation} showed the existence of $N_d$ and gave the bounds $d+\log_2 d<N_d<d^2/4+2d+4$.

In this paper, we introduce another notation $s(n)$ for given $n$, that is the smallest integer $k$ such that any permutation of length $n$ can be reconstructed from its $k$-deck, or equivalently, any two permutations of length $n$ have different $k$-decks.
We give lower and upper bounds of $s(n)$ for large $n$, \[3^{0.811\times\log_3^{1/2}(n+1)}\le s(n)\leq 2\lceil\sqrt{(n-2)\ln(n-3)}\rceil +2,\] which imply much better bounds of $N_d$, \[ d+\exp(\Omega(\sqrt{\ln d}))<N_d<d+O(\sqrt{d\ln d})\] for large $d$.  We also provide a feasible algorithm which assists to determine the exact values of $s(n)$ for $n\leq 10$.



\subsection{Related work}

The problem of reconstructing a sequence from the multiset (i.e., the $k$-deck) of  all its subsequences of length $k$,  was introduced by Kalashnik \cite{kalashnik1973reconstruction} in an information-theoretic study about deletion channels. The main task is to determine  the minimum number $k$ such that one can reconstruct any binary sequence of length $n$ from its $k$-deck. The best known bounds of the minimum $k$ are  $ \exp{(\Omega\sqrt{\ln n})}$ \cite{foster2000improvement} and  $O(\sqrt{n})$ \cite{dudik2003reconstruction}. This deck problem has been extended to matrices~\cite{kos2009reconstruction} and general higher dimensions~\cite{zhong2025reconstruction}, for which the minimum number $k$ such that one can reconstruct any binary $d$-dimensional hypermatrix of order $n$ from its $k$-deck, the multiset of its sub-hypermatrices of order $k$, is $O(n^{\frac{d}{d+1}})$.

The problem of partition reconstruction is to ask for which $n$ and $k$ one can uniquely determine any partition
of $n$ from its set of $k$-minors \cite{mnukhin1993combinatorial,cameron1996stories}. Here,
a $k$-minor of a partition $\lambda$ of a positive integer $n > k$ is a partition of $n-k$ whose Young diagram
fits inside that of $\lambda$.
Monks \cite{monks2009solution}  showed that partitions of $n\ge k^2 + 2k$ are uniquely determined by their sets of $k$-minors, which is best possible. Cain and Lehtonen \cite{cain2022reconstructing} completely characterized the standard Young tableaux that can be reconstructed from their sets or multisets of $1$-minors.

For permutation reconstruction, Gouveia and Lehtonen \cite{gouveia2021permutation} showed that every permutation of length $n\ge 5$ is reconstructible from any  $\lceil n/2\rceil + 2$ of its $(n-1)$-minors.  Reconstruction of permutations from other types of minors is also well studied.  Monks \cite{monks2009reconstructing} showed that any permutation of $[n]$ can be reconstructed from its set of \emph{cycle minors} if and only if $n\ge 6$. Lehtonen \cite{lehtonen2015reconstructing} showed that every permutation of a finite set with at least five elements is reconstructible from its identification minors. De Biasi \cite{de2014permutation} proved that the problem of reconstructing a permutation given the absolute differences of consecutive entries is NP-complete.

\subsection{Organization}
The paper is organized as follows. In Section~\ref{notations}, we give necessary definitions and notations, then give a lower bound  of $s(n)$. In Section \ref{upper}, we analyze the relationship between the reconstructibility and some specific functions, and give an upper bound of $s(n)$. In Section \ref{small}, we study  bounds and exact values of $s(n)$ for small $n$ by giving an efficient algorithm, and deduce much better bounds of $N_d$ than that in \cite{raykova2006permutation} from our bounds of $s(n)$.

\section{Notations}\label{notations}

  For positive integers $n$ and $n_1< n_2$, let $[n]:=\{1,2,\ldots,n\}$ and $[n_1,n_2]:=\{n_1,n_1+1,\ldots, n_2\}$. For two sequences $\vx$ and $\vy$, let $\vx\mid \vy$ be the concatenation of them.

 Let $\S_n$ be the set of all permutations on $[n]$.
For a permutation $\vx\in\S_n$, a \emph{$k$-minor} of $\vx$ is a subsequence of  $\vx$ with length $k$ whose entries are renumbered by $1,2,\ldots,k$ preserving order. For example, deleting the first entry of $\vx=25134\in \S_5$, we get a subsequence $5134$ of length four. Then renumber the entries by $\{1,2,3,4\}$ preserving order, we have a $4$-minor $4123$ of $\vx$.
The \emph{$k$-deck} of $\vx$, denoted by $D_{k}(\vx)$, is the multiset of all $k$-minors of $\vx$. If a permutation $\vx\in\S_n$ can be uniquely determined by $D_{k}(\vx)$, we say $\vx$ is \emph{$k$-reconstructible}. Moreover, if $D_k(\vx)=D_k(\vy)$ for two different permutations $\vx,\vy$ in $\S_n$, we say $\vx$ and $\vy$ are $k$-equivalent, and write $\vx\overset{k}{\sim}\vy$. 

\begin{example}\label{eg1}
For $\vx=13524\in \S_5$, $D_4(\vx)=\{2413,1423,1324,1243,1342\}$. It can be verified by computer that different permutations in $\S_5$ have different $4$-decks, so $\vx$ is $4$-reconstructible. However, one can check that $13524\overset{3}{\sim} 14253$, that is, $\vx$ is not $3$-reconstructible.
\end{example}

It is easy to see that for any two $\vx, \vy\in \S_n$,  $D_l(\vx)=D_l(\vy)$ implies that $D_k(\vx)=D_k(\vy)$ for any $k\leq l$. That is, $\vx\overset{l}{\sim}\vy$ implies that $\vx\overset{k}{\sim}\vy$ for any $k\leq l$. So if $\vx$ is $k$-reconstructible, then $\vx$ is $l$-reconstructible for any $k<l\leq n$.
%
By this fact, it is interesting to consider the following problem.

\textbf{Q}:  Given $n$, determine the least integer $k\leq n$ such that any $\vx\in \S_n$ is $k$-reconstructible. Denote this number by $s(n)$.
\bigskip

By Example~\ref{eg1}, $s(5)=4$.  The parameter $s(n)$ is an increasing function of $n$ by the following result. 

\begin{proposition}\label{snincr}
For any positive integer $n$, we have $s(n)\le s(n+1)$.
\end{proposition}
\begin{proof}
It is clear that $s(1)=1,s(2)=2$. So $s(n)\le s(n+1)$ holds for $n=1$. Assume that $n\ge 2$. Since all permutations in $\mathcal{S}_n$ have the same $1$-deck, we have $s(n)\ge 2$. Let $k=s(n)-1\ge 1$. By the definition of $s(n)$, there exist two distinct permutations $\vx,\vy\in \S_n$ such that $D_k(\vx)=D_k(\vy)$, and hence $D_{k-1}(\vx)=D_{k-1}(\vy)$. Then for two distinct permutations $\vx\mid(n+1),\vy\mid(n+1)\in \S_{n+1}$, we have
\begin{equation*}
  \begin{aligned}
  D_k(\vx\mid(n+1))=&D_k(\vx)\uplus \{\vw\mid  k:\vw\in D_{k-1}(\vx)\}\\
  =& D_k(\vy)\uplus \{\vw\mid  k:\vw\in D_{k-1}(\vy)\}= D_k(\vy\mid (n+1)).
  \end{aligned}
\end{equation*}
That is, $\vx\mid(n+1)$ and $\vy\mid(n+1)$ have the same $k$-deck, which implies that $s(n+1)\ge k+1=s(n)$. This completes the proof.
\end{proof}

\bigskip
In order to provide a lower bound of $s(n)$, we establish a   mapping between  binary sequences and  permutations, then  adopt the result in \cite{dudik2003reconstruction} for sequence reconstructions. For a sequence $\vp\in\{0,1\}^n$, let $D'_k(\vp)$ be its $k$-deck, i.e., the multiset of all its subsequences of length $k$.
   Let $s'(n)$ denote  the smallest integer $k$ such that any  sequence in $\{0,1\}^n$ can be reconstructed from its $k$-deck, or equivalently, any two binary sequences of length $n$ have different $k$-decks.

\vspace{0.3cm}

\begin{theorem}\label{seq-perm}
For any positive integer $n$, $s(n)\ge s'(n)$.
\end{theorem}
\begin{proof}
Let $k=s'(n)-1$. Then there exist two sequences $\vp,\vq\in\{0,1\}^n$ satisfying $D'_k(\vp)=D'_k(\vq)$. It suffices to construct two distinct permutations $\vx,\vy\in \S_n$ such that $D_k(\vx)=D_k(\vy)$, then by the definition of $s(n)$, we have $s(n)>k=s'(n)-1$, i.e., $s(n)\ge s'(n)$.

Define a mapping $\Psi$ from $\{0,1\}^n$ to $\S_n$ as follows.
Suppose $\vp\in\{0,1\}^n$ has $m$ ones. 
Then $\Psi$ maps $\vp$ to a permutation $\vx\in \S_n$ by changing the $m$ ones in $\vp$ into $[m]$ with increasing order and preserving index positions, and  changing the $(n-m)$ zeros in $\vp$ into $[m+1,n]$ with increasing order and preserving index positions. For example, if $\vp=0010011$, then $\vx=4516723$. It is easy to see that, every  $k$ positions in $[n]$ gives a $k$-subsequence   $\vw$ of $\vp$ and a $k$-minor $\vz\in \S_k$ of $\vx$ with $\vz=\Psi(\vw)$.

Let $\vy=\Psi(\vq)$.
Then $D'_k(\vp)=D'_k(\vq)$ implies that $D_k(\vx)=D_k(\vy)$. So we have $s(n)\ge s'(n)$.
\end{proof}


Known values of $s'(n)$  for small $n$  from \cite{dudik2003reconstruction,foster2000improvement} are listed below. Comparing with  Table~\ref{s(n)} in Section~\ref{small}, we known that $s(n)>s'(n)$ in general.
{\color{blue}
\begin{table*}[h!]
\center
\[\begin{array}{cccccccccccc}
\hline
\text{$n$} & 1 & 2 & 3 & 4 & 5 & 6 & 7 & 8 &  9 &  10 & 11 \\
\text{$s'(n)$} & 1 & 2 & 2 & 3 & 3 & 3 & 4 & 4 & 4 & 4 &  4 \\
\hline
\\\hline
\text{$n$} & 12 & 13 & 14 & 15 & 16 & 17 & 18 & 19 & 20 & 21 & 22\\
\text{$s'(n)$} & 5 & 5 & 5 & 5 & [5,8] & [5,8] & [5,8] & [5,8] & [5,8] & [5,8] & [5,8] \\
\hline
\end{array}\]
\caption{Values of $s'(n)$  for small $n$. $[\cdot,\cdot]$ means lower and upper bounds.}
\end{table*} }



In \cite{dudik2003reconstruction},  the authors considered a dual parameter of $s'(n)$: Given $k$, let $S(k)$ denote the least integer $n> k$ such that there exist distinct binary sequences $\vp,\vq\in \{0,1\}^n$ with the same $k$-deck. By \cite{dudik2003reconstruction},
$$S(k)\leq 1.2\Gamma(\log_3 k)\times 3^{3/2\log_3^2 k-1/2\log_3 k}$$ for $k\ge 5.$ By assist of Mathematica, we obtain the following  result for $s(n)$.

\vspace{0.3cm}
\begin{theorem}\label{lower_bound}
$s(n)\geq 3^{0.811\times\log_3^{1/2}(n+1)}$, $n\ge 16$.
\end{theorem}
\begin{proof} Let $f(k)=\lfloor1.2\Gamma(\log_3 k)3^{3/2\log_3^2 k-1/2\log_3 k}\rfloor$  for $k\ge 5$, which is the upper bound of $S(k)$. Then $f(k)$ is a non-decreasing function, and $f(5)=16$. Now for any integer $n$, let $k$ be the integer satisfying $f(k)\leq n < f(k+1)$. By definition, there exist distinct sequences $\vp,\vq\in \{0,1\}^n$ with  the same $k$-deck. This means $s'(n)\geq k+1$.

Let $k= 3^{0.811\times \log_3^{1/2}(n+1)}$, we have  $f(k)\leq n$ which can be verified by Mathematica.  By Theorem~\ref{seq-perm}, we obtain the lower bound
\[s(n)\geq s'(n)\geq 3^{0.811\times\log_3^{1/2}(n+1)}.\]
\end{proof}

\vspace{0.3cm}


\bigskip

\section{An Upper Bound of $s(n)$}\label{upper}
In this section, we provide an upper bound of $s(n)$, where the idea is from \cite{krasikov1997reconstruction} working on reconstruction of sequences.

For a permutation $\vz=z_1\ldots z_k\in \S_k$, define its indicator function of $(i,j)$-order as:
\begin{equation}
	z_{ij}=\left\{
	\begin{aligned}
		1 & , & z_i<z_j\\
		0 & , & z_i>z_j
	\end{aligned}
	\right. ~~~~1\leq i<j\leq k.
\nonumber
\end{equation}
Given a permutation $\vx=x_1\ldots x_n\in\S_n$  and its $k$-deck $D_k(\vx)$, let $S_{ij}(\vx)$ be the total sum of $(i,j)$-orders of all minors in $D_k(\vx)$. That is, $S_{ij}(\vx)=\sum_{\vz\in D_k(\vx)}z_{ij}$, $1\leq i<j\leq k$. The next lemma relates the value of $S_{ij}(\vx)$ to the orders of $\vx$.


\begin{lemma}\label{lesj1j2}

For any $1\leq i<j\leq k$,
\[S_{ij}(\vx)=\sum_{1\leq i'<j'\leq n}\binom{i'-1}{i-1}\binom{j'-i'-1}{j-i-1}\binom{n-j'}{k-j}x_{i'j'}.\]

\end{lemma}
\begin{proof}
For any $\vz\in D_k(\vx)$, suppose that $z_i,z_j$ are originally $x_{i'}$, $x_{j'}$ for some $i'$ and $j'$, respectively. Write
$$\vx=\underbrace{\ldots}_{i'-1}~~~\emph{$x_{i'}$}~\underbrace{\ldots}_{j'-i'-1}~~~\emph{$x_{j'}$}~\underbrace{\ldots}_{n-j'}.$$
  Then the number of such $\vz\in D_k(\vx)$ with $z_i,z_j$ originally from $x_{i'}$, $x_{j'}$ is
  \[\binom{i'-1}{i-1}\binom{j'-i'-1}{j-i-1}\binom{n-j'}{k-j}.\]
Then $S_{ij}(\vx)$ can be obtained by summing up them multiplied by $x_{i'j'}$. 
\end{proof}

Suppose two permutations $\vx=x_1\ldots x_n,\vy=y_1\ldots y_n$ in $\S_n$  satisfy $\vx\overset{k}{\sim}\vy$. Then $D_k(\vx)=D_k(\vy)$ and thus $S_{ij}(\vx)=S_{ij}(\vy)$ for $1\leq i<j\leq k$. Define $\delta_{ij}:=x_{ij}-y_{ij}\in \{0,\pm1\}$ for $1\leq i<j\leq n$. We have the following equalities by Lemma \ref{lesj1j2},

\begin{equation}\label{eq-ident}
  \sum_{1\leq i'<j'\leq n}\binom{i'-1}{i-1}\binom{j'-i'-1}{j-i-1}\binom{n-j'}{k-j}\delta_{i'j'}=0, \quad 1\leq i<j\leq k.
\end{equation}

Consider bivariate polynomials $f_{ij}(x,y)=\binom{x-1}{i-1}\binom{y-x-1}{j-i-1}\binom{n-y}{k-j}$, $1\leq i<j\leq k$.  Note that $\deg f_{ij}=k-2$ and $f_{ij}(x,y)=0$ for $0\le x<i$, or $y-x<j-i$, or $n-k+j<y\le n$.

\begin{lemma}\label{lemfij}
For fixed integers $n\geq k\geq 1$, the set $\{f_{ij}\}_{1\leq i<j\leq k}$ is a basis for the space of bivariate polynomials of degree at most $k-2$.
\end{lemma}
\begin{proof}
Consider $\varphi(x,y)=\sum_{1\leq i<j\leq k}\mu_{ij}f_{ij}(x,y)=\sum_{1\leq i<j\leq k}\mu_{ij}\binom{x-1}{i-1}\binom{y-x-1}{j-i-1}\binom{n-y}{k-j}$. It suffices to show that $\varphi(x,y)$ is not identically zero whenever the coefficients $\mu_{ij}$ are not all zero. Assume  $\mu_{i_0j_0}$ is the first nonzero coefficient in lexicographical order, that is, whenever $i<i_0$ or $i=i_0$ and $j<j_0$, $\mu_{ij}=0$. Then $\varphi(i_0,j_0)=\mu_{i_0j_0}\binom{n-j_0}{k-j_0}$, since $\binom{i_0-1}{i-1}=0$ for $i>i_0$ and $\binom{j_0-i_0-1}{j-i_0-1}=0$ for $i=i_0$ and $j>j_0$. Hence $\varphi(i_0,j_0)\neq 0$, that is, $\varphi(x,y)$ is not identically zero.
\end{proof}
\vspace{0.3cm}

Combining Lemma~\ref{lemfij} and Eq. (\ref{eq-ident}), we obtain the following necessary condition for $\vx\overset{k}{\sim}\vy$.

\vspace{0.3cm}
\begin{corollary}\label{k_equ}
If $\vx\overset{k}{\sim}\vy$ in $\S_n$, then for any bivariate polynomial $\varphi(x,y)$ of degree at most $k-2$,
\[\sum_{1\leq x<y\leq n}\delta_{xy}\varphi(x,y)=0.\]
\end{corollary}

For easy application, we obtain the following sufficient condition for the reconstructibility by changing bivariate polynomials in \Cref{k_equ} to  univariate polynomials.

\vspace{0.5cm}
\begin{theorem}\label{poly2}
If there exists a polynomial $\phi(x)$ satisfying $\deg \phi\leq k-2$ and $\phi(-1)>(n-2)\sum_{x=0}^{n-3}|\phi(x)|$, then any permutation $\vx\in\S_n$ is $k$-reconstructible. 
\end{theorem}
\begin{proof}
Assume on the contrary that there exist $\vx\neq \vy\in \S_n$ with $\vx\overset{k}{\sim}\vy$. For $1\leq i\leq n-1$, define $$\delta_i=\sum_{i<j\leq n}\delta_{ij}=\sum_{i<j\le n}(x_{ij}-y_{ij})=\sum_{i<j\le n}x_{ij}-\sum_{i<j\le n}y_{ij}.$$
Let $m$ be the first integer in $[n-1]$ satisfying $x_m\neq y_m$. Note that $\sum_{i<j\le n}x_{ij}$ is  the number of $j>i$ satisfying $x_j>x_i$. So we have  $\sum_{i<j\le n}x_{ij}= \sum_{i<j\le n}y_{ij}$ for all $i\in [m-1]$, and $\sum_{m<j\le n}x_{mj}\neq \sum_{m<j\le n}y_{mj}$. Thus $\delta_i= 0$ for $i\in [m-1]$ and $\delta_m\neq 0$. Without loss of generality, assume $x_m< y_m$, then $\delta_m\geq 1$.

For $i\in [m+1, n-1]$, $|\delta_i|\leq \max\{\sum_{i<j\le n}x_{ij},\sum_{i<j\le n}y_{ij}\}  \leq  n-i\leq n-2$. Define a bivariate polynomial $\varphi(x,y)=\phi(x-m-1)$. Then $\deg\varphi=\deg\phi\leq k-2$ and
\begin{equation*}
\begin{aligned}
\sum_{1\leq x<y\leq n}\delta_{xy}\varphi(x,y)&=\sum_{x=1}^{n-1}\phi(x-m-1)\sum_{y=x+1}^{n}\delta_{xy}=\sum_{x=1}^{n-1}\phi(x-m-1)\delta_x\\
&=\delta_m\phi(-1)+\sum_{x=m+1}^{n-1}\phi(x-m-1)\delta_x\geq \phi(-1)-(n-2)\sum_{x=0}^{n-3}|\phi(x)|>0.
\end{aligned}
\end{equation*}
However, by Corollary \ref{k_equ}, $\sum_{1\leq x<y\leq n}\delta_{xy}\varphi(x,y)=0$, which is a contradiction.
\end{proof}

\vspace{0.5cm}
It remains to construct a polynomial $\phi(x)$ satisfying the conditions in \Cref{poly2}. The construction is based on \cite[Section 3]{foster2000improvement}.

Take $\phi(x)=p_k^2(x),~p_k(x)=\sum_{i=0}^{k}\frac{T_i(-1)}{d_i}T_i(x)$ from \cite[Theorem 3.2]{foster2000improvement}, where $T_i$ is the special case of the Hahn polynomials \cite{karlin1961hahn},\cite{nikiforov1991classical},
$$T_i(x)=\sum_{j=0}^i(-1)^j\frac{\binom{i}{j}\binom{i+j}{j}}{\binom{n}{j}}\binom{x}{j},$$
and
$$d_i=\frac{n+i+1}{2i+1}\frac{\binom{n+i}{n}}{\binom{n}{i}}.$$ Then $p_k$ is of degree $k$, and $\deg\phi=2k$. For fixed $i,j\in [k]$, $T_i$ and $T_j$ are orthogonal on $[0,n]$, i.e.,
\begin{equation}\label{eqtij}
  \sum_{x=0}^n T_i(x)\cdot T_j(x)=\delta_{i,j}d_i.
\end{equation}

Define $\Delta(\phi):=(n+1)\sum_{x=0}^{n}|\phi(x)|-\phi(-1)$  and $F_k:=\sum_{i=0}^{k}T_i^2(-1)/d_i$. By Eq. (\ref{eqtij}), we have
\begin{equation}\label{eqfk}
\begin{aligned}
\Delta(\phi)&=(n+1)\sum_{x=0}^n p_k^2(x)-p_k^2(-1)\\
&=(n+1)\sum_{x=0}^n \left(\sum_{i=0}^{k}\frac{T_i(-1)}{d_i}T_i(x)\right)^2-\left(\sum_{i=0}^{k}\frac{T_i^2(-1)}{d_i}\right)^2\\
&=(n+1)\sum_{i=0}^{k}\frac{T_i^2(-1)}{d_i}-\left(\sum_{i=0}^{k}\frac{T_i^2(-1)}{d_i}\right)^2\\
&=(n+1-F_k)F_k.
\end{aligned}
\end{equation}
Hence when $F_k> n+1$, we have $\Delta(\phi)< 0$.

Based on the above discussions, we provide an upper bound of $s(n)$ as below.

\vspace{0.3cm}
\begin{theorem}\label{upper_bound}
$s(n)\leq 2\lceil\sqrt{(n-2)\ln(n-3)}\rceil+2$ when $n\ge 7$.
\end{theorem}
\begin{proof} It is equivalent to show that $s(n+3)\leq 2\lceil\sqrt{(n+1)\ln n}\rceil+2$ when $n\ge 4$. Let $k=\lceil\sqrt{(n+1)\ln n}\rceil$.
  By \Cref{poly2}, it suffices to show the existence of a polynomial $\phi$ of degree $2k$ satisfying $\phi(-1)>(n+1)\sum_{x=0}^{n}|\phi(x)|$. Let $\phi(x)=p_k^2(x)$, which is of degree $2k$. By Eq. (\ref{eqfk}), we only need to prove that $F_k> n+1$.
%

By the conclusion of \cite[Section 3.2.1]{foster2000improvement}, we have
 $$F_k>e^{(k+1)k/(n+1)}-1+1/(n+1).$$
Since $k\geq \sqrt{(n+1)\ln n}$, we have
\begin{equation*}
\begin{aligned}
e^{(k+1)k/(n+1)}-1+1/(n+1)&> e^{\ln n+\sqrt{\ln n/(n+1)}}-1\\
&=ne^{\sqrt{\ln n/(n+1)}}-1\\
&> n(1+\sqrt{\ln n/(n+1)})-1>n+1.
\end{aligned}
\end{equation*}
Here, the last inequality is true when $n\geq 4$. Hence $F_k> n+1$.
\end{proof}


\vspace{0.3cm}

\section{Some specific values for small $n$}\label{small}
In this section, we consider exact values of $s(n)$ for small $n$ by assistant of computers. 

For each $\vx\in \S_n$, let $S_j(\vx)$ be the total number of  minors in $D_{k}(\vx)$ with the $j$-th coordinate being  $1$,  $j\in [k]$. For each  $t\in [n-k+1]$, let $i_t(\vx)\in [n-t+1]$ be the location that $t$ lies in $\vx$ after deleting $1,2,\ldots, t-1$. For example, when $\vx=13524$ and $k=2$, $i_1(\vx)=1$, $i_2(\vx)=3$, $i_3(\vx)=1$ and $i_4(\vx)=2$.
 The following lemma relates $S_j(\vx)$  to  values of $i_t(\vx)$. We write $i_t$ instead of $i_t(\vx)$ if there is no confusion.

%

\begin{lemma}\label{chara}

For each $\vx\in \S_n$ and $j\in [k]$, we have
\begin{equation}
S_j(\vx)=\sum_{t=1}^{n-k+1}\binom{i_t-1}{j-1}\binom{n-i_t-(t-1)}{k-j}.
\end{equation} 
\end{lemma}
\begin{proof}
It is easy to see that only symbols $1,2,\ldots,n-k+1$ in $\vx$ may result in a symbol $1$ in some minor in $D_{k}(\vx)$, since larger numbers cannot be changed to $1$.

For each $t\in[n-k+1]$, the value $i_t(\vx)$ indicates that after deleting all symbols from $[t-1]$ in $\vx$, there are exactly $i_t-1$ symbols appear before the symbol $t$. Keeping the symbol $t$, the symbol $t$ becomes $1$ in the minor.
To make this $1$  in the $j$-th position in some minor, we need to delete $(i_t-j)$ more symbols before $t$, see
$$\underbrace{\cdots}_{i_t-1}~~~\emph{t}\underbrace{\cdots}_{n-i_t-(t-1)}.$$
So the number of minors of length $k$ with the $j$-th position being $1$  is $\binom{i_t-1}{i_t-j}\binom{n-i_t-(t-1)}{n-k-(i_t-j)-(t-1)}=\binom{i_t-1}{j-1}\binom{n-i_t-(t-1)}{k-j}$. 
Summing over $t\in [n-k+1]$, we get the conclusion.
\end{proof}
\vspace{0.3cm}

Denote $\vi(\vx):=(i_1,i_2,\ldots,i_{n-k+1})\in [n]\times [n-1]\times \cdots \times[k]$, and call it the characteristic  vector of $\vx$. By Lemma~\ref{chara}, the value $S_j(\vx)$ only depends on the characteristic  vector of $\vx$.  Different permutations may have the same characteristic  vector. Define the real location of $t$ in $\vx=x_1x_2\ldots x_n$ as $\zeta_t$, that is, $x_{\zeta_t}=t$. Call the vector $\zeta(\vx):=(\zeta_1,\zeta_2,\ldots,\zeta_{n-k+1})$ the location vector of $\vx$.  Clearly, the  characteristic  vector and the location vector can be deduced from each other recursively. See the following lemma.

\vspace{0.3cm}

\begin{lemma}\label{real_location} For each $t\in [n-k+1]$, we have $i_t=\zeta_t-\sum_{j<t} 1_{\{\zeta_t> \zeta_j\}}$ and $\zeta_t=i_t+\sum_{j<t} 1_{\{i_t\ge \zeta_j\}}$.

\end{lemma}
\begin{proof}The former equalities are obvious.  We only prove the latter ones.
It is easy to see that $\zeta_1=i_1$.
For $\zeta_2$, if $i_2<\zeta_1$,   deleting $1$  does not affect the position of $2$, so $\zeta_2=i_2$; if $i_2\ge \zeta_1$,   deleting $1$  will result in  one step forward on $2$,
thus $\zeta_2=i_2+1$.
Continuing with this argument, we give
$\zeta_t=i_t+\sum_{j<t} 1_{\{i_t\ge \zeta_j\}}.$
\end{proof}

 In all the above arguments, we consider the position of symbol $1$ in the minors.
Symmetrically, we can consider the position of symbol $k$ in minors.  Define $\bar{S}_j(\vx)$  the number of  minors in $D_{k}(\vx)$ with the $j$-th coordinate being  $k$. For each $t\in [k,n]$, let $\bar{i}_{t}(\vx)\in [t]$  be the location that $t$ lies in $\vx$ after deleting $n,n-1,\ldots, t+1$. By the similar proof as for Lemma~\ref{chara}, we have the following result.


\begin{corollary}\label{charac}
For each $\vx\in \S_n$ and $j\in [k]$, we have
\begin{equation}
\bar{S}_j(\vx)=\sum_{t=k}^{n}\binom{\bar{i}_{t}-1}{j-1}\binom{t-\bar{i}_t}{k-j}.
\end{equation}
\end{corollary}

As in Lemma~\ref{real_location}, $\bar{i}_{t},t\in [k,n]$ can also  determine $\zeta_{t},t\in [k,n]$. In fact,  $\zeta_n=\bar{i}_n$, and \begin{equation}\label{eqzet}
              \zeta_{t}=\bar{i}_{t}+\sum_{j>t} 1_{\{\bar{i}_{t}\ge \zeta_{j}\}}
            \end{equation} for $t\in [k,n-1]$. 
           Call $\bar{\vi}(\vx)=(\bar{i}_n,\bar{i}_{n-1},\ldots,\bar{i}_{k})$  the reverse characteristic vector and $\bar{\zeta}(\vx):=(\zeta_n,\zeta_{n-1},\ldots,\zeta_{k})$ the reverse location vector.
            Consequently, for two permutations $\vx\neq \vy$, if $\vi(\vx)=\bar{\vi}(\vy)$, then $\zeta(\vx)=\bar{\zeta}(\vy)$.


\vspace{0.3cm}
%
\vspace{0.3cm}

Now we are ready to apply the above arguments to design an exhaustive algorithm: input $n$ and $k$, and output a pair $(\vx,\vy)\in \S_n\times \S_n$ which meets the conditions $S_j(\vx)=S_j(\vy)$ and $\bar{S}_j(\vx)=\bar{S}_j(\vy)$ for all $j\in [k]$. By the fact that $\vx\overset{k}{\sim} \vy$ implies $S_j(\vx)=S_j(\vy)$ and $\bar{S}_j(\vx)=\bar{S}_j(\vy)$ for $j\in [k],$
if the algorithm returns nothing, it means that any two permutations in $\S_n$ are not equivalent, that is, $s(n)\leq k$.  In fact, our algorithm tries to output a set of permutations in $\S_n$ that are possible to satisfy the above conditions.


\vspace{0.3cm}
\textbf{Algorithm 1}:

Step 1. For each vector $\vi=(i_1,i_2,\ldots, i_{n-k+1})\in [n]\times [n-1]\times \ldots \times [k]$, we compute the corresponding $S_j$ value for each $j\in [k]$ by Lemma \ref{chara}. These values of $S_j$ are stored as a row in a matrix $\mathbf{S}_{l\times k}$ with row index $\vi$, where $l=\frac{n!}{(k-1)!}$. The row $\vi$ of $\mathbf{S}$ corresponds to a set of permutations in $\S_n$  with the same   characteristic  vector  $\vi$ and thus the same $S_j$ values.

Step 2. By pairwise comparing all rows in $\mathbf{S}$, we collect all pairs $(\vi,\vi')$ of rows such that $\mathbf{S}(\vi,j)=\mathbf{S}(\vi',j)$ for each $j\in [k]$. We store all such pairs $(\vi,\vi')$ as a row in an array $\mathbf{P}$ with two columns. Note that $(\vi',\vi)$ is also recorded in $\mathbf{P}$.

Step 3. Considering $\bar{\vi}=(\bar{i}_n,\bar{i}_{n-1},\ldots, \bar{i}_{k})\in [n]\times [n-1]\times \ldots \times [k]$ and values $\bar{S}_j$, we can get exactly the same matrices $\mathbf{S}$ and  $\mathbf{P}$. This means that each row $\vi$ of $\mathbf{S}$ can be viewed as the $S_j$ values of some permutations $\vx$ with $\vi(\vx)=\vi$, and can also be viewed as the $\bar{S}_j$ values of some other permutations $\vx'$ with $\bar{\vi}(\vx')=\vi$.

For each entry $\vi$ in the first column of $\mathbf{P}$, by considering it as a characteristic vector $(i_1,i_2,\ldots, i_{n-k+1})$, we compute a location vector $\zeta(\vx)=(\zeta_1,\zeta_2,\ldots,\zeta_{n-k+1})$ by Lemma~\ref{real_location} for some $\vx$ with $\vi(\vx)=\vi$; by considering it as a reverse characteristic vector $(\bar{i}_n,\bar{i}_{n-1},\ldots, \bar{i}_{k})$, we compute a reverse location vector $\bar{\zeta}(\vx')=(\zeta_n,\zeta_{n-1},\ldots,\zeta_{k})$ by Eq. (\ref{eqzet}) for some $\vx'$ with $\bar{\vi}(\vx')=\vi$.

Step 4. Suppose that $(\vx,\vy)\in \S_n\times \S_n$ satisfies $S_j(\vx)=S_j(\vy)$ and $\bar{S}_j(\vx)=\bar{S}_j(\vy)$ for all $j\in [k]$.  Then there exist two different rows $\vi$ and $\bar{\vi}$, such that  $\vi(\vx)=\vi$ and $\bar{\vi}(\vx)=\bar{\vi}$. Then the location vector $\zeta(\vx)=(\zeta_1,\zeta_2,\ldots,\zeta_{n-k+1})$ computed by $\vi(\vx)=\vi$ and the reverse location vector $\bar{\zeta}(\vx)=(\zeta_n,\zeta_{n-1},\ldots,\zeta_{k})$ computed by $\bar{\vi}(\vx)=\bar{\vi}$ should match each other to form one permutation.

 By definition of $\mathbf{P}$, we could simply pairwise check entries in the first column of $\mathbf{P}$, and output all possible candidates $\vx$. Otherwise, stop and return ``No solution!''.


\vspace{0.3cm}

\begin{remark} If Algorithm 1 returns a set of permutations, we need to continue to pairwise check whether they are equivalent. However, Algorithm 1 is quite useful when it returns no solution, which implies $s(n)\leq k$. In this case, it is very efficient since the data is based on all different vectors $\vi\in [n]\times [n-1]\times \ldots \times [k]$ instead of all different permutations $\vx\in \S_n$, and the algorithm runs  just by computing and comparing a few characteristic values,  which greatly reduces the  computation space  and computation time.

%
\end{remark}
\vspace{0.3cm}

\begin{lemma}\label{alg910}
For $(n,k)\in \{(9,5),(10,5)\}$, Algorithm 1 returns ``No solution!''. So $s(9)\leq 5$ and $s(10)\leq 5$.
\end{lemma}

\subsection{A table for $s(n)$}

Now we determine or bound values of $s(n)$ for small $n$. We first introduce a related notation studied in \cite{smith2006permutation,raykova2006permutation}.
Given $d$, let $N_d$ be the smallest integer such that for any $n\geq N_d$, we can reconstruct permutations of length $n$ from their $(n-d)$-deck. That is, for all $n\geq N_d$, $D_{n-d}(\vx)\neq D_{n-d}(\vy)$ for different  $\vx,\vy\in\S_n$. Smith \cite{smith2006permutation} proved that $N_1=5, N_2=6$ and $N_3\leq 13$.  Raykova \cite{raykova2006permutation} showed that the existence of $N_d$ for all $d$, and further proved that $N_3=7$ and $N_4\ge 9$, then gave upper and lower bounds by \[d+\log_2 d<N_d<d^2/4+2d+4.\]

By definition, if $N_d\leq n$, that is, for any two different $\vx,\vy\in S_n$, $D_{n-d}(\vx)\neq D_{n-d}(\vy)$, then $s(n)\leq n-d$. So by $N_1=5, N_2=6$ and $N_3=7<8$, we have $s(5)\leq 4,s(6)\leq 4,s(7)\leq 4$ and $ s(8)\leq 5$. By $N_d<d^2/4+2d+4$, $s(d^2/4+2d+4)\leq d^2/4+d+4$ for every $d$. But this is much weaker than our upper bound $s(n)=O(\sqrt{n\ln n})$ in Theorem~\ref{upper_bound}.

If  $N_d=n+1$, that is, there exist two different $\vx,\vy\in S_n$ such that $D_{n-d}(\vx)= D_{n-d}(\vy)$, then $s(n)>n-d$, that is, $s(N_d-1)>N_d-1-d$. So by $N_2=6=5+1$ and $N_3=7=6+1$, we have $s(5)>3$, $s(6)>3$.


Combining the above results, we have $s(5)=s(6)=4$. By $1247356\overset{3}\sim 1263475$ and $68573142\overset{4}{\sim} 75862413$, we have $s(7)=4$ and $s(8)=5$. By Proposition~\ref{snincr}, $s(n)\geq s(8)=5$ when $n\geq 9$. Combining Lemma~\ref{alg910}, we have $s(9)=s(10)=5$.
For $11\le n\le 14$, since $n\ge N_3$,  $s(n)\le n-3$;  by \Cref{snincr}, $s(n)\ge s(10)=5$. For $15\le n\le 19$, since $N_4\leq 4^2/4+2\times4+4-1=15\leq n$, we have $s(n)\le n-4$. For $n\ge 20$, since $N_5<5^2/4+2\times 5+4=20.25$, which means $n\ge N_5$, we have $s(n)\le n-5$. The lower bound of $s(n)$, $16\le n\le 22$ can be obtained from \Cref{lower_bound}.

Finally we list the best  bounds of $s(n)$ for $n\leq 22$ in  Table~\ref{s(n)},  where the first unknown value is $s(11)$.
\begin{table*}[h!]
\center
\[\begin{array}{cccccccccccc}
\hline
\text{$n$} & 1 & 2 & 3 & 4 & 5 & 6 & 7 & 8 &  9 &  10 & 11\\
\text{$s(n)$} & 1 & 2 & 3 & 4 & 4 & 4 & 4 & 5 & 5 & 5 &  [5,8]\\
\hline
\\\hline
\text{$n$} & 12 & 13 & 14 & 15 & 16 & 17 & 18 & 19 & 20 & 21 & 22\\
\text{$s(n)$} & [5,9] & [5,10] & [5,11] & [5,11] & [6,12] & [6,13] & [6,14] & [6,15] & [6,15] & [6,16] & [6,17] \\
\hline
\end{array}\]
\caption{Bounds of $s(n)$ for small $n$.}\label{s(n)}
\end{table*}

%
%

\subsection{Improvements on $N_d$}

Based on our bounds of $s(n)$, we are able to improve bounds of $N_d$ asymptotically.

If $s(n)> k$, that is, there exist two different $\vx,\vy\in S_n$ such that $D_{k}(\vx)= D_{k}(\vy)$,  then $N_{n-k}>n$. Then by Theorem~\ref{lower_bound}, that is, $s(n)=\exp(\Omega(\sqrt{\ln n}))$, we have
\begin{equation}\label{ndlb}
  N_d>d+\exp(\Omega(\sqrt{\ln d})),
\end{equation}
when $d$ is large, which is much stronger than $N_d >d+\log_2 d$ in \cite{raykova2006permutation}.


For the upper bound, we improve as follows.
\begin{theorem}\label{ndub}
   $N_d\leq d+3\sqrt{d\ln d}$  when $d$ is large enough.
\end{theorem}
\begin{proof}
  Let  $n_0=d+3\sqrt{d\ln d}-2$.  First, we claim that $(n_0-d-2)^2>4n_0\ln (n_0-1)$ when $d$ is large enough. This is true because $(n_0-d-2)^2=9d\ln d(1-o(1))$ and $4n_0\ln (n_0-1)= 4d\ln d(1+o(1))$ when $d$ goes to infinity.

  For each fixed $d$ satisfying $(n_0-d-2)^2>4n_0\ln (n_0-1)$, we next claim that $(n-d-2)^2>4n\ln (n-1)$ for any $n\geq n_0$. This is true since the left hand side grows much faster than the right hand side.

  Finally, we show that for each $n\geq n_0+2= d+3\sqrt{d\ln d}$, any permutation in $\S_{n}$ is $(n-d)$-reconstructible, i.e., $s(n)\leq n-d$. In fact, by \Cref{upper_bound}, we have
  \[s(n)\leq 2\lceil\sqrt{(n-2)\ln(n-3)}\rceil+2< 2\sqrt{(n-2)\ln(n-3)}+4<(n-2)-d-2+4=n-d. \]
  This completes the proof.
\end{proof}

Note that \Cref{ndub} greatly improves the upper bound $N_d<d^2/4+2d+4$ from \cite{raykova2006permutation} asymptotically.  Combining Eq. (\ref{ndlb}) and \Cref{ndub}, we have for large $d$,

\begin{equation*}
  d+\exp(\Omega(\sqrt{\ln d}))<N_d<d+O(\sqrt{d\ln d}).
\end{equation*}

\section{Conclusion}
By applying results for sequence reconstruction, we prove that the least integer $k$ such that any permutation in $\mathcal{S}_n$ is $k$-reconstructible is between $\exp{(\Omega(\sqrt{\ln n}))}$ and $O(\sqrt{n\ln n})$.
As a consequence, we improve the bounds significantly for the well-studied parameter $N_d$, the smallest integer such that any permutation in $\mathcal{S}_n$ with $n\geq N_d$ is $(n-d)$-reconstructible.  The new bounds are $ d+\exp(\Omega(\sqrt{\ln d}))<N_d<d+O(\sqrt{d\ln d})$ asymptotically, and the previous best bounds were  $d+\log_2 d<N_d<d^2/4+2d+4$ in \cite{raykova2006permutation}.

\end{document}